\documentclass{amsart}
  \usepackage{amssymb}
\renewcommand{\qed}{}
\title{A new order theory of set systems and better quasi-orderings}

\author{Yohji AKAMA}
\address{Mathematical Institute, Tohoku University\\
Sendai Miyagi JAPAN, 980-8578}
\email{akama@m.tohoku.ac.jp}
\date{\today}

\begin{document}\newtheorem{theorem}{Theorem}
\newtheorem{definition}{Definition}
\newtheorem{example}{Example}
\newtheorem{fact}{Fact}
\newtheorem{condition}{Condition}
\newtheorem{corollary}{Corollary}
\newtheorem{lemma}{Lemma}

\newtheorem{q}{Conjecture}
\newtheorem{proposition}{Proposition}
\def\ae{\mathrel{\preceq^\forall_\exists}}

\def\closure#1{\mathrm{ss}\left({#1}\right)}

\def\emptyseq{\langle\,\rangle}

\def\ewprod{\mathrel{\widetilde{\times}}}
\def\ewunion{\mathrel{\widetilde{\cup}}}

\def\I{\mathcal{I}}
\def\inj{\mathsf{i}}
\def\inj{}
\def\injinv{\mathsf{i}^{-1}}
\def\injinv{}

\def\inhfin#1{\left[{#1}\right]^{<\omega}}

\def\finclass#1{{#1}^{<\omega}}
\def\Finclass#1{\left({#1}\,\right)^{<\omega}}

\def\l{\langle}
\def\L{\mathcal{L}}
\def\Langle{\left\langle}
\def\M{\mathcal{M}}

\def\op{\ensuremath{\mathfrak{O}}}  %% Symbol for Operators

\def\PF#1{\mathrm{PF}\left({#1}\right)}
\def\ppf#1{\left|\mathrm{PF}\left({#1}\right)\right|}
\def\Prd#1{\mathrm{Prod}(#1)}
\def\R{\right\rangle}
\def\SGL{\mathcal{S}\textit{ingl}\,}

\def\Zset{\mathbb{Z}}

\def\r{\rangle}

\def\ot{\mathrm{o.t.}}
\def\otp{\mathrm{otp}}
\def\power#1{P\left(#1\right)}

\def\Ram{\mathrm{Ra}}
\newcommand{\SetSys}{\ensuremath{\mathbb{SS}}}

\def\X{\mathcal{X}}

\def\Y{\mathcal{Y}}

\def\qo#1{\mathrm{qo}\left(#1\right)}

%\newcommand{\midc}{\;|\;}

%maketitle は abstract と keyword の後に入れてください。

%
%  英文アブストラクト (論文ページに出力されます) をここに書きます。
%
\begin{abstract}%
By reformulating a learning process of a set system $L$ as a game
 between Teacher~(presenter of data) and Learner~(updater of the abstract
 independent set of the data), we define the order type $\dim L$ of $L$ to be the
 order type of the game tree. The theory of this new order type and
 continuous, monotone function between set systems corresponds to the
 theory of well quasi-orderings (WQOs). As Nash-Williams developed the
 theory of WQOs to the theory of better quasi-orderings (BQOs), we
 introduce a set system that has order type and corresponds to a BQO. We
 prove that the class of set systems corresponding to BQOs is closed by
 any monotone function.  In (Shinohara and Arimura. ``Inductive
 inference of unbounded unions of pattern languages from positive
 data.'' {\em Theoretical Computer Science}, pp. 191--209, 2000), for
 any set system $L$, they considered the class of arbitrary (finite)
 unions of members of $L$. From viewpoint of WQOs and BQOs, we
 characterize the set systems $L$ such that the class of arbitrary
 (finite) unions of members of $L$ has order type. The characterization
 shows that the order structure of the set system $L$ with respect to
 the set-inclusion is not important for the resulting set system having
 order type.
 We point out continuous, monotone function  of set systems is similar to
 positive reduction to Jockusch-Owings' weakly semirecursive sets. 
\end{abstract}

\maketitle
%\tableofcontents
\section{Introduction}

A \emph{set system} $\L$ over a set $T$, a subfamily of the powerset
$P(T)$, is a topic of (extremal)
combinatorics~\cite{MR866142}\cite{MR1931142}, as well as a target of an
algorithm to learn in computational learning theory~\cite{Lange2008194}.  

A \emph{well quasi-ordering}~\cite{MR0049867}~(\textsc{wqo} for short)
is, by definition, a quasi-ordering $(X,\, \preceq)$ which has neither an infinite
antichain nor an infinite descending chain. \textsc{Wqo}s are
employed in algebra~\cite{MR0049867},
combinatorics~\cite{MR0111704}\cite{MR2595703}, formal language
theory~\cite{MR1196555}\cite{dalessandro08:_well_quasi_order_in_formal_languag_theor}\cite{ehrenfeucht83:_regul_of_contex_free_languag}\cite{MR2083928},
and so on. 

\textsc{Wqo}s and related theorems such as Higman's
theorem~\cite{MR0049867}, K\"onig's lemma and Ramsey's
theorem~\cite{MR866142} are sometimes employed in computational learning
theory.  In
\cite{kanazawa98:_learn_class_of_categ_gramm}\cite{114871}\cite{93373},
sufficient conditions for set systems being learnable is studied with
K\"onig's lemma and Ramsey's theorem, In \cite{shinohara:inductive}, for
a set system $\L$, Shinohara-Arimura considered the \emph{unbounded
unions} of $\L$, that is, the class $\finclass{\L}$ of nonempty finite
unions of members of $\L$, and then they used Higman's theorem to study
a sufficient condition for it being learnable.  In
\cite{brecht09:_topol_and_algeb_aspec_of}, de Brecht employed
\textsc{wqo}s to calibrate \emph{mind change} complexity of unbounded
unions of restricted pattern languages. Motivated by
\cite{kanazawa98:_learn_class_of_categ_gramm}\cite{114871}\cite{93373},
a somehow systematic study on the relation between \textsc{wqo}s and a
class of learnable set systems is done in \cite{AkamaSetSystem}, as
follows:

(i) By reformulating a learning process of a set system $\L$ as a game
 between Teacher~(presenter of data) and Learner~(updater of abstract
 independent set), we define the order type $\dim \L$ of $\L$ to be the
 order type of the game tree, if the tree is well-founded. According to
 computational learning theory, if an indexed family $\L$ of recursive
 languages has well-defined $\dim \L$ then $\L$ is learnable by an
 algorithm from positive data.  If a set system has the well-defined
 order type, then we call it a \emph{finitely elastic set
 system}~(\textsc{fess} for short).  See Definition~\ref{def:dimfess}.
 \bigskip

(ii) For each \emph{quasi-ordering} $\X=(X,\, \preceq)$, we consider the
set system $\closure{\X}$ consisting of upper-closed subsets of $X$. The
set system has the order type equal to the \emph{maximal order
type}~\cite{MR0447056} of $\X$. Furthermore, the construction
$\closure{\bullet}$ has an left-inverse $\qo{\cdot}$. Here for a set
system $\L$, $\qo{\L}$ is a quasi-ordering $(\bigcup\L,\ \preceq_\L)$
such that
\begin{align*}
x\preceq_\L y \iff\forall L\in \L,\left(x\in L\implies y\in L\right) .
\end{align*}
The maximal order type $\otp(\X)$ of $\X$ is defined if and
only if $\X$ is a \textsc{wqo}. 
For any quasi-ordering $\X$, if one of
     $\otp(\X)$ and $\dim \closure{\X}$ is defined then the other side
	 is defined with the same ordinal number.
So
\textsc{fess}s correspond to \textsc{wqo}s.
\bigskip

(iii) 
For every nonempty set $U$, the product topological space $\{0,1\}^U$ of
the discrete topology $\{0,1\}$ is called a
\emph{Cantor space}. A subspace of a Cantor space is represented by $\L,
\M,\ldots$. We say a function from $\M$ to $\L$ is \emph{continuous}, if
it is continuous with respect to the subspaces $\M,\L$ of the Cantor spaces.
We identify $\{0,1\}^U$ with the powerset $P(U)$, and a
function from $\{0,1\}^U$ to $\{0,1\}^U$ with a function from $P(U)$ to
	 $P(U)$. We say $\op:\M\to\L$ is a \emph{deformation}, if it is monotone~(i.e. $M\subseteq M'$ implies $\op(M)\subseteq\op(M')$.)
If a deformation is continuous, then it has following finiteness condition:

\begin{lemma}\label{lem:11}Let $\op : \{0,1\}^{\bigcup\M} \to
\L$.
\begin{enumerate}
\item $\op$
is a deformation, if and only if there is a binary relation $R \subseteq  \left(\bigcup\L\right)
      \times P\left(\bigcup\M\right)$ such that
\begin{align}
&\forall M\in \M\, \forall x\in \bigcup\L\nonumber\\ 
&\ \ \bigl(\op( M ) \ni  x\; \Leftrightarrow\;  \exists v
 \subseteq M.\;  R(x, v)\bigr). \label{eq:mono}
\end{align}

\item
 $\op$ is a \emph{continuous} deformation, if and only if there exists $R
      \subseteq \left(\bigcup\L\right) \times P\left(\bigcup\M\right)$
      such that \eqref{eq:mono} holds, but $v$ is a \emph{finite} set
      whenever $R(x,v)$ holds, and there are only \emph{finitely} many
      such $v$'s for each $x$.  (\cite{AkamaSetSystem})
\end{enumerate} 
For each binary relation $R\subseteq\left(\bigcup \L\right) \times
 P(\bigcup \M)$, the function $\op$ satisfying \eqref{eq:mono} is
 unique.  So we write it by $\op_R$. Conversely, every deformation
 $\op:\{0,1\}^{\bigcup\M}\to\L$ is written as $\op_R$ by a binary relation
\begin{align*}
R:=\bigl\{(x,v)\in \left(\bigcup\L\right)\times P\left(\bigcup\M\right)\;;\;
 \op(v)(x)=1\bigr\}.
\end{align*}
\end{lemma}

The class of \textsc{wqo}s is closed under finitary operations such
	 as Higman embedding~\cite{MR0049867} and  topological minor relation~\cite[Sect.~1.7]{MR2744811} between finite
	 trees~\cite[Ch.~12]{MR2744811}\cite{MR0111704}. The class of finite graphs is a
	 \textsc{wqo} under the \emph{minor relation}. 
	 Robertson-Seymour's proof of it is given in the numbers IV-VII,
	 IX-XII and XIV-XX of their series of over 20 papers under the
	 common title of \emph{Graph Minors}, which has been appearing
	 in the \emph{Journal of Combinatorial Theory, Series B}, since
	 1983. For a shorter proof, see
 recent papers by Kawarabayashi and
	 his coauthors.

The class of  \textsc{fess}s enjoys a useful closure condition:

\begin{proposition}[\cite{AkamaSetSystem}]\label{prop:u}
For any set systems $\L$ and $\M$ and any continuous deformation
 $\op:\inj \{0,1\}^{\bigcup\M}\to \inj \L$, if $\M$ is an \textsc{fess},
 so is the image $\injinv \op\left[\inj \M\right]$ of $\M$ by $\op$.
\end{proposition}

By it, we prove that for  various  (nondeterministic) language
operators~(e.g. Kleene-closure, shuffle-product~\cite{MR1231799}\cite{MR2567276}, shuffle-closure~\cite{MR2078698}, (iterated) literal
shuffle~\cite{berard87:_liter_shuff}, union, product,
intersection), the elementwise application of such operator to (an)
\textsc{fess}(s) induces an \textsc{fess}. 
 
Roughly speaking, a deformation transforms any quasi-ordering $\preceq$
to the \emph{powerset ordering}~\cite{MR1884426} $\ae$. It is through
our correspondence $(\closure{\bullet},\ \qo{\cdot})$ between
quasi-orderings to set systems~(Section~\ref{sec:contbqo}).  Although
the powerset ordering of Rado's \textsc{wqo}~\cite{MR0066441} is not a
\textsc{wqo}~\cite[Corollary~12]{MR1884426}, the class of \emph{better
quasi-orderings}~\cite{MR0221949}~(\textsc{bqo}s for short) is closed
with respect to the powerset ordering. There are \emph{infinitary}
operations under which the class of \textsc{wqo}s is not closed but the
class of \textsc{bqo}s is.  So we introduce a \emph{better elastic set
system}~(\textsc{bess} for short), as a set system corresponding to a
\textsc{bqo}.  We show that the class of \textsc{bess}s is closed under
the image of any deformation~(Section~\ref{sec:contbqo}), where any
deformations are infinitary in a sense of Lemma~\ref{lem:11}. By this
and (i), we can develop the computational learning theory of
$\omega$-languages~\cite{267871}.

The notion of \textsc{bess} is useful in investigating a following set system
\begin{align}
\finclass{\L}:=\{ \bigcup \M\;;\; \emptyset\ne\M\subseteq \L,\ \#\M<\infty\}. \label{fess:unbounded}
\end{align}
It is studied for the learnability of language classes such as a class
of regular pattern languages~\cite{shinohara:inductive}. We characterize
the set systems $\L$ such that $\finclass{\L}$ are again \textsc{fess}s,
and then we prove that for every \textsc{bess} $\L$, $\finclass{\L}$ is
an \textsc{fess}.

We remark that another importance of set system $\finclass{\L}$.
 We conjecture that the order type of an \textsc{fess} is the
supremum, actually the maximum, of the order types of the ``linearizations'' of the \textsc{fess}.
This conjecture corresponds to a proposition useful in investigating
       \textsc{wqo}s:
\begin{proposition}[\protect{\cite{MR0447056}}]
The order
       type of a \textsc{wqo} is the maximum order type of the
       linearizations of the \textsc{wqo}.
\end{proposition}
A ``linearization'' of an \textsc{fess} $\L$ seems to
be a subfamily of $\finclass{\L}$.

We hope that our study on \textsc{fess}s and \textsc{bess}s are useful
in solving problems~(e.g. decision problem of timed
Petri-nets~\cite{MR1946092}\cite{MR2050792}, the multiplicative exponential
linear logic~\cite{Groote:2004:VAT:1018438.1021843}) which are related
to \textsc{wqo}s and \textsc{bqo}s but hard to solve with conventional
arguments for \textsc{wqo}s and \textsc{bqo}s.

The rest of paper is organized as follows:
In the next section, 
we recall the powerset ordering and  Marcone's characterization~\cite{MR1884426} of
\textsc{wqo}s such that the powerset orderings are again \textsc{wqo}s.
We also review the combinatorial definition of a
\textsc{bqo}. Then we recall that the class of \textsc{bqo}s is closed with respect
to the powerset ordering.
In Section~\ref{sec:contbqo}, for the set system $\closure{\preceq}$ of
upper-closed sets of a fixed quasi-ordering $\preceq$, the image by a
deformation is essentially the set system of upper-closed sets of the
powerset ordering of $\preceq$. Then we introduce a \textsc{bess} as  an
\textsc{fess} that corresponds to a \textsc{bqo}. Then we prove that the
image of a \textsc{bess} by any deformation is an \textsc{fess}.
In Section~\ref{sec:linearization}, we characterize the class of set
systems $\L$ such that $\finclass{\L}$ is an \textsc{fess}, from
viewpoint of \textsc{bqo} theory.  We contrast our characterization
with Shinohara-Arimura's sufficient condition~\cite{shinohara:inductive}
for a set system $\L$ to have an \textsc{fess} $\finclass{\L}$.

In appendix, we propose to extend Ramsey numbers to estimate the ordinal
order type of set systems, and then present miscellaneous results for
computable analogue of (iii). Finally, we review the order type of a set
system from \cite{AkamaSetSystem}.

\section{Better quasi-orderings and powerset ordering}\label{sec:bqo}

A better quasi-ordering~(\textsc{bqo} for short), a stronger concept than a
\textsc{wqo}, has pleasing closure properties with respect to
\begin{itemize}
\item
embedding for transfinite sequences~\cite{MR0221949};
\item topological minor relation for infinite
trees 

\cite{MR1816801}\cite{MR0175814}; and

\item  a \emph{powerset ordering} $\ae$ (see Definition~\ref{def:ae})
(\cite[Corollary~10]{MR1884426}).
\end{itemize}

\subsection{Combinatorial definition of BQOs}
We first recall the definition of \textsc{bqo}s by
\emph{barriers}~\cite{MR1884426} and then the closure of \textsc{bqo}s
with respect to the powerset ordering $\ae$.
Please be advised to refer \cite{MR1219735}\cite{MR1884426}
 for the detail. 

Hereafter the first infinite ordinal $\omega$ is identified with the set
of nonnegative integers.  The class of subsets $X$ of $U$ such that the
cardinality of $X$ is less than $\alpha$ (equal to $\alpha$, resp.) is
denoted by $[U]^{<\alpha}$ ($[U]^{\alpha}$, resp.). A set
$X\subseteq\omega$ is often identified with the sequence enumerating it
in a strictly increasing order.
\begin{definition}
\begin{enumerate}
\item
We say $B\subseteq \inhfin{\omega}$ a \emph{barrier}, if 
(1) $\bigcup B$ is
 infinite; (2) for all $\sigma \in [\bigcup B]^\omega$ there exists $s\in B$
 such that $s$ is a prefix of $\sigma$; and
(3) for all $s,t\in B$, $s\not\subseteq t$.

\item For $s,t\in \inhfin{\omega}$, we write $s \triangleleft t$, if,
 the sequence $s$ is a prefix of the sequence
$u=s\cup t$ and the sequence $t$ is a prefix of $u\setminus\{\min u\}$.

\item Let $\ot(B)$ the maximal order type of $B$ with respect to the
     lexicographical ordering.
\end{enumerate}
\end{definition}

Observe that 
\begin{align}
\SGL:=\{\ \{n\}\ ;\ n\in \omega\}\label{sgl}
\end{align}
is a barrier. Any barrier $B$ of $\ot(B)$ being $\omega$ consists only
of singletons, according to \cite[p.~342]{MR1884426}.

We recall an $\alpha$-\textsc{wqo} and a
\textsc{bqo}~\cite[Definition~3]{MR1884426}.

\begin{definition}\label{def3}
Let $\alpha$ be a countable ordinal and $\preceq$ a quasi-ordering on
 $Q$. We say a function $f:B\to Q$ is \emph{good} with respect
 to $\preceq$, if there are some $s,t\in B$ such that $s\triangleleft t$
 and $f(s)\preceq f(t)$. Otherwise we say $f$ is \emph{bad}.
We say $\preceq$ an $\alpha$-\textsc{wqo}, if for every barrier
 with $\ot(B)\le\alpha$  every function $f:B\to Q$ is good with
 respect to $\preceq$. If $\preceq$ is an
 $\alpha$-\textsc{wqo} for all countable ordinal $\alpha$, we call
 $\preceq$ a \textsc{bqo}.\end{definition}

Because $\SGL$ is a barrier, every \textsc{bqo} is a \textsc{wqo}. When
we define an $\alpha$-\textsc{wqo} for
 a countable ordinal $\alpha$, we have only to consider
 only \emph{smooth barriers} among the barriers,
as  \cite{MR1884426} explains:

\begin{definition}
By a \emph{smooth} barrier, we mean a barrier $B$ such that for all
 $s,t\in B$ with $\#s<\#t$  there exists $i\le \#s$ such that the $i$-th
 smallest element of $s$ is less than that of $t$.
\end{definition}

By an \emph{indecomposable
ordinal}, we mean $\omega^\beta$ such that $\beta$ is any ordinal. Recall \cite[Corollary~3.5]{MR1219735}. 
\begin{proposition}\label{cor3.5}
If $B$ is a smooth barrier, then the ordinal $\ot(B)$ is indecomposable.\end{proposition} 

Then we have \cite[Theorem~4]{MR1884426}.
\begin{proposition}\label{thm4}
Let $\alpha$ be a countable ordinal and $\preceq$ a quasi-ordering on
 $Q$. $\preceq$ is an $\alpha$-\textsc{wqo} if and only if for every smooth barrier
 $B$ with $\ot(B)\le \alpha$, every map $f:B\to Q$ is \emph{good} with respect
 to $\preceq$.
 \end{proposition}
\begin{corollary}\label{cor:a} $\preceq$ is a \textsc{bqo} if and only if it is an
$\alpha$-\textsc{wqo} for all countable infinite indecomposable ordinal
$\alpha$.\end{corollary}

We use properties of a barrier from \cite[Lemma~6]{MR1884426}.
 
\begin{proposition}\label{lem6}
For a barrier $B$, let $B^2\subseteq\inhfin{\omega}$ be
\begin{displaymath}
B^2:=\{ s\cup t\;;\; s,t\in B,\ s\triangleleft
t\}.\end{displaymath}
Then 
\begin{enumerate}
%\item \label{a} $B^2$ is a block;
\item \label{b} for each $t\in B^2$ there exist unique $\pi_0(t),\pi_1(t)\in B$
      such that $\pi_0(t)\triangleleft \pi_1(t)$ and $t=\pi_0(t)\cup
      \pi_1(t)$;
\item \label{c} if $t,t'\in B^2$ and $t\triangleleft t'$ then
      $\pi_1(t)=\pi_0(t')$;
\item \label{e}  $B^2$ is a barrier; and
\item \label{f} if  $\ot(B)$ is indecomposable then
      $\ot(B^2)=\ot(B)\cdot\omega$.
\end{enumerate}
\end{proposition}

\subsection{Powerset ordering}\label{subsec:po}

\begin{definition}\label{def:ae}
For a quasi-ordering  $(X, \ \preceq)$, we define a quasi-ordering on
the powerset $P(X)$ by 
\begin{displaymath}
v \ae v':\iff  \forall x'\in v'\exists x\in v.\ x\preceq x'.
\end{displaymath}
\end{definition} 

It is studied for the reachability
analysis of Petri nets~(verification of infinite-state
systems~\cite{MR1946092}, Timed Petri
net~\cite{MR1946092}\cite{MR2050792},$\ldots$.)

\begin{proposition}[\protect{\cite{MR1737752}, \cite[Theorem~9]{MR1884426}}]\label{thm9}
If $\alpha$ is a countable infinite indecomposable ordinal and
 $(Q,\preceq)$ is an $(\alpha\cdot\omega)$-\textsc{wqo}, then the
 powerset $P(Q)$ ordered by $\ae$ is an
 $\alpha$-\textsc{wqo}.\end{proposition}

\begin{lemma}\label{lem:rado} For a quasi-ordering
 $\X=(Q,\preceq)$, the following are equivalent:
\begin{enumerate}
\item \label{prop:rado:1} (a) $\X$ is a \textsc{wqo}; and (b) let $F$ be
      any function from $[\omega]^2$ to $Q$. Then if
      $F(\{i,j\})\prec F(\{i,j+1\})$ for any $i<j<\omega$, then there are
      $i<j<k<\omega$ such that $F(\{i,j\})\prec F(\{j,k\})$.

\item \label{prop:rado:2}
 $\X$ is an $\omega^2$-\textsc{wqo}.

\item \label{prop:rado:3}
 $\left(\inhfin{Q}, \ \ae\right)$ is a \textsc{wqo}.

\item \label{prop:rado:4}
 $\left(\power{Q}, \ \ae\right)$ is a \textsc{wqo}.
\end{enumerate}
\end{lemma}
\begin{proof} The proof is similar to that of
 \cite[Corollary~12]{MR1884426}, but we use Rado's characterization~\cite[Theorem~3]{MR0066441} of
 \textsc{wqo}s $(Q, \preceq)$
such that the set of sequences of elements of $Q$ with the length
 $\omega$ is again a \textsc{wqo}. For any \textsc{wqo} $\X=(Q,\preceq)$ and any countable ordinal $\alpha$,
 let $\X^\alpha$ ($\X^{<\alpha}$ resp.) be the set of sequences of elements of $Q$ of length
 $\alpha$ (less than $\alpha$, resp.)  quasi-ordered by naturally generalized Higman's
 embedding. The condition \eqref{prop:rado:1} is equivalent to 
 $\X^\omega$ being a \textsc{wqo}, by \cite[Theorem~3]{MR0066441}.
 By Higman's theorem, it is equivalent to
 $\left(\X^\omega\right)^{<\omega}= \X^{<\omega^2}$ being a
 \textsc{wqo}.  By
\cite{MR1219735}, it is equivalent to 
 \eqref{prop:rado:2}. The equivalence to the other conditions follows
 from \cite[Corollary~12]{MR1884426}. \qed
\end{proof}

\section{Better elasticity | deformation as powerset ordering}
\label{sec:bess}
\label{sec:contbqo}

Given a deformation $\op$ and a \textsc{wqo}
$\X$,
we try to construct explicitly from $\X$ a suitable \textsc{wqo}
$\Y$  such that
$\op[\closure{\X}]\subseteq\closure{\Y}$.

\begin{definition}\label{def:sqsubseteq} Let $\L$ and $\M$ be set systems.
Suppose  $R\subseteq (\bigcup\L) \times \power{\bigcup\M}$ and 
 $\X:=(\bigcup\M,\ \preceq)$ is a quasi-ordering.  Then 
define a quasi-ordering $Q_R(\X)$ by $(\bigcup\L,\ \sqsubseteq)$ as follows:
For any $x,x'\in\bigcup\L$,
we write $x\sqsubseteq x'$, if whenever $R(x,v)$ holds, there exists
 $v'$ such that $R(x',v')$ and $v\ae v'$.
\end{definition}

\begin{lemma}\label{lem:qo}Let $\L$ and $\M$ be set systems.
Suppose  $R\subseteq (\bigcup\L) \times \power{\bigcup\M}$. If $\X$ is a quasi-ordering on $\bigcup\M$, then $Q_R(\X)$ is indeed a quasi-ordering on $\bigcup\L$.
\end{lemma}
\begin{proof}Let $Q_R(\X)$ be $(\bigcup\L, \sqsubseteq)$.
It is easy to see that $x\sqsubseteq x$. From  $x\sqsubseteq x'$ and
 $x'\sqsubseteq x''$, we  derive $x\sqsubseteq x''$. Let $R(x, v)$. By $x\sqsubseteq x'$, there exists $v'$ such that $R(x',v')$
 and $v\ae v'$. By $x'\sqsubseteq x''$, there exists $v''$ such that
 $R(x'',v'')$ and $v'\ae v''$. Because $\ae$ is a quasi-ordering, we have $v\ae
 v''$. \qed\end{proof}

It seems difficult to replace a ``quasi-ordering'' with a ``partial ordering'' in
Lemma~\ref{lem:qo}. For a following theorem, see (ii) in the first section for a left-inverse $\qo{\cdot}$ of
$\closure{\bullet}$, and Lemma~\ref{lem:11} for the definition of $\op_R$.

\begin{theorem}\label{lem:b}For any
 $R\subseteq(\bigcup\L)\times\power{\bigcup\M}$ , we have $\op_R\left[ \M \right]\subseteq\closure
{Q_R(\qo{\M})}$.\end{theorem}

\begin{proof}  Let $(\bigcup\M,\ \preceq)=\qo{\M}$ and
$(\bigcup\L,\ \sqsubseteq)=Q_R(\qo{\M})$.

We verify if $A\in\M$ and $\op_R(A)\ni x \sqsubseteq x'$, then
$\op_R(A)\ni x'$.

By $\op_R(A)\ni x$, there exists $v$ such that $R(x,v)$ and $v\subseteq
 A$. Since $x \sqsubseteq x'$, there exists $v'$ such that
 $R(x',v')$ and every $x'\in v'$ has $x\in v$ with $x\preceq x'$.
Because $x\in A$ and $A\in \M$
 is upper-closed with respect to $\preceq$, $x'\in A$. So  $v'\subseteq A$. This means
 $\op_R(A)\ni x'$.  \qed
\end{proof} 

We are not sure whether $Q_R(\qo{\closure{\X}})$ is a \textsc{wqo} for all
\textsc{wqo} $\X$, because, according to Lemma~\ref{lem:rado},
the quasi-ordering $\ae$ is not always a \textsc{wqo}.

Instead, we introduce a stronger set system than an
\textsc{fess}.
\begin{definition}[\textsc{BESS}s]We say a set system $\L\subseteq P(T)$ is
 a \emph{better elastic set system}~(\textsc{bess} for short) or $\L$ has \emph{better elasticity}, provided 
$\qo{\L}$ is a \textsc{bqo}.
\end{definition}

\begin{example}
 For every \textsc{wqo} $\X$ which is not a \textsc{bqo}, a set system
     $\closure{\X}$ is an \textsc{fess} but not a \textsc{bess}, since
     $\qo{\cdot}$ is a left-inverse of $\closure{\bullet}$.
\end{example}

\begin{lemma}
\begin{enumerate}
\item
 A quasi-ordering $(X,\, \preceq)$ is a \textsc{bqo}, if and only if
     $\mathrm{ss}\left((X,\,\preceq)\right)$ is a \textsc{bess}.

\item Every \textsc{bess} is an \textsc{fess}.
\end{enumerate} 
\end{lemma}
\begin{proof}(1) By (ii) in the first section.
(2) If a set system $\L$ has an infinite learning sequence~(see Definition~\ref{def:dimfess}) $\left\langle \langle t_0,
 L_1\rangle , \langle t_1, L_2\rangle,\ldots \right\rangle$, then we
 have a barrier $\SGL$ and a function $f\;;\;
\{i\}\in \SGL\mapsto t_i$  such that $t_i\in L_j\not\ni t_j$ hence
 $t_i\not\preceq_\L t_j$ $(i<j)$. This contradicts that $\qo{\L}$ is a
 \textsc{bqo}. \qed
\end{proof}

As the class of \textsc{bqo}s is closed under infinitary operations than 
the class of \textsc{wqo}s is, we prove that the class of
\textsc{bess}s is closed under infinitary operations~(i.e., deformations) than the class of
\textsc{fess}s is. See Lemma~\ref{lem:11} for the characterization of deformations.

\begin{theorem}\label{main:bqo}
Assume $\L$ and $\M$ are set systems
 and  $\op:\inj \M\to \inj \L$
 is a  deformation. Then if $\M$ is a \textsc{bess},
 so is $\injinv  \op\left[\inj \M\right] $.
\end{theorem}

To prove Theorem~\ref{main:bqo}, we have only to verify a following
lemma, in view of Corollary~\ref{cor:a}:

\begin{lemma}\label{lem:main}Let $\op$ be a
deformation from a set system $\M$ to a set system $\L$. 
If $\alpha$ is a countable infinite indecomposable ordinal, and $\qo{\M}$ is an
 $\alpha\cdot\omega^2$-\textsc{wqo}, then $\qo{\op[\M]}$ is an
 $\alpha$-\textsc{wqo}.\end{lemma}

 \begin{proof} Write $\op$ as $\op_R$ for some relation $R\subseteq
  (\bigcup\L)\times P\left( \bigcup\M\right)$.
 Let the quasi-ordering $\qo{\M}$ be  $(\bigcup\M,\, \preceq)$ and 
$Q_R(\qo{\M})$  be $(\bigcup\L,\
  \sqsubseteq)$, as in Definition~\ref{def:sqsubseteq}. 
Assume $\qo{\op[\M]}$ is not an $\alpha$-\textsc{wqo}. 
By Theorem~\ref{lem:b} and (ii) in the first section,
\begin{align*}
\qo{\op\left[\M\right]}\supseteq \qo{\closure{Q_R(\qo{\M})}} =
  Q_R(\qo{\M}).
\end{align*}  
Since $\qo{\op[\M]}$ is not an $\alpha$-\textsc{wqo}, the quasi-ordering
$Q_R(\qo{\M})=(\bigcup\L,\ \sqsubseteq)$ is neither.  By Proposition~\ref{thm4}, there are smooth
barrier $B$ of $\ot(B)\le\alpha$ and a function $f:B\to\bigcup\op[\M]$
such that for all $u,v\in B$, $u\triangleleft v$ implies
$f(u)\not\sqsubseteq f(v)$. By Proposition~\ref{lem6}~\eqref{b}, for all
$t\in B^2$, we have $\pi_0(t), \pi_1(t)\in B$ and $\pi_0(t)\triangleleft
\pi_1(t)$. So $f\left(\pi_0(t)\right)\not\sqsubseteq
f(\left(\pi_1(t)\right))$ for all $t\in B^2$.  By
Definition~\ref{def:sqsubseteq}, for some $v$ with $R\left(
f\left(\pi_0(t)\right),\, v\right)$ and for all $v'$, if $ R\left(
f\left(\pi_1(t)\right),\, v'\right)$ then $ v\not\ae v'$.  For each $t$,
let $g(t)$ be one of such $v$. Then $g$ is a function from $B^2$ to
$P(\bigcup\M)$.

  Then for all $t,t\in B^2$, we have $g(t)\not\ae g(t')$ whenever
  $t\triangleleft t'$. In other words, the function  $g$  is bad with
  respect to the powerset ordering $\ae$. To verify it, 
  first recall
 $\pi_1(t)=\pi_0(t')$ by Proposition~\ref{lem6}~\eqref{c}. By the definition  of $g$, we have $g(t)\not\ae
 v'$ whenever
 $R\left(f\left(\pi_1(t)\right),\ v'\right)$. Moreover $R\left(
  f(\pi_0(t')),\ g(t')\right)$. Because Proposition~\ref{lem6}~\eqref{c}
  implies $\pi_1(t)=\pi_0(t')$, we have
  $R\left(f\left(\pi_1(t)\right),\, g(t')\right)$. Therefore, $g(t)\not\ae g(t')$.

By Proposition~\ref{cor3.5} and Proposition~\ref{lem6}~\eqref{f},
 $\ot(B^2)=\ot(B)\cdot\omega\le \alpha\cdot\omega$. Because $B^2$ is a
 barrier by Proposition~\ref{lem6}~\eqref{e}, $(P(\bigcup\L),\ \ae)$ is not an
 $(\alpha\cdot\omega)$-\textsc{wqo}, which contradicts
 Proposition~\ref{thm9}.\qed
 \end{proof}

A following immediate corollary of
Theorem~\ref{main:bqo} may be useful in developing computational learning
theory of $\omega$-languages~\cite{267871}: Let  $\Sigma^\infty$ be a set of possibly
 infinite sequences of elements in an alphabet $\Sigma$, and $\L$ be a set system over
 $\Sigma^\infty$. The concatenation operation of two sequences is
 defined similarly as that of two finite sequences except that
for an infinite sequence $u$ and
 sequence $v$, the concatenation $u v$ is defined as $u$.
For  $L\subseteq\Sigma^\infty$, 
the \emph{$\omega$-closure}~\cite{267871} $L^\omega$ of $L$
 is the set of infinitely iterated
 concatenation $u_1 u_2 u_3 \cdots$ of sequences $u_1,u_2,\ldots\in L$.

\begin{corollary}[Closure of $\omega$-languages] 
If $\L$ is a \textsc{bess},  so are following
 classes:
\begin{enumerate}
\item $\{ L^\omega\;;\; L\in \L\}$.
\item 
$\{ L^\mathrm{sh}\;;\; L\in \L\}$. Here $L^\mathrm{sh}$ is the
      shuffle-closure 
$\{\varepsilon\}\cup L\cup (L\diamond L)\cup((L\diamond
L)\diamond L) \cup\cdots$, and 
$L\diamond L'$ is the set of $  u_1 v_1 u_2 v_2 \cdots$ such that
$u_1 u_2 \cdots \in L$, $v_1 v_2 \cdots \in L'$ and  $u_i, v_i \in
      \Sigma^*$ ($i\ge 1$).
\end{enumerate}
\end{corollary}

\section{Linearizations of set systems and powerset orderings}\label{sec:linearization}

Many study on \textsc{wqo}s use de Jongh-Parikh's
 theorem~\cite{MR0447056}:
``The order-type
$\otp(\X)$ of a \textsc{wqo} $\X$ is the \emph{maximum} of order-types of the
linearizations of $\X$.'' 

We wish to require that if a linear order $\Y$ is a linearization of a
quasi-ordering of $\X$, then $\closure{\Y}$ is a `linearization' of a
$\closure{\X}$. Then a `linearization' of a set system $\L\subseteq P(T)$
should be a set system $\M\subseteq P(T)$ linearly ordered by set
inclusion and $\M$ consists of unions of members
of $\L$ such that for any $L\in \L$ there exists a subfamily
$\L'\subseteq\L$ such that $L\in \L'$ and $\bigcup\L'\in \M$.

\begin{lemma}\label{lem:linear}For any $\textsc{fess}$ $\L$, if $\qo{\L}$ is
 a \textsc{wqo} then the class $\hat{\L}:=\{ \ \bigcup \M\ ;\  \emptyset\ne\M\subseteq \L\
\}$ containing $\L$ closed under arbitrary unions is an \textsc{fess}. 

\end{lemma} 
\begin{proof}Since each member $\bigcup\M\in \hat\L$ is upper-closed
 with respect to the \textsc{wqo} $\qo{\L}$, any learning sequences~(see Definition~\ref{def:dimfess}) of the set system $\hat{\L}$ are
 those of the set system $\closure{\qo{\L}}$. By the premise $\qo{\L}$ is a \textsc{wqo}, so 
$\closure{\qo{\L}}$ is an \textsc{fess}. Therefore any learning
 sequence should be finite. Hence $\hat \L$ is an \textsc{fess}.
\qed\end{proof}

\begin{q}Define a suitable  linearization of a set system.
Do we have
\begin{align*}
 \dim \X = \max\left\{ \dim \Y\;;\; \Y\ \mbox{is a linearization of
 }\X\right\}
\end{align*}
for any set system $\X$ ?
\end{q}

We characterize the set systems $\L$ such that the set system of
arbitrary (finite) unions of members of $\L$ is an \textsc{fess}, from
viewpoint of \textsc{wqo}s, \textsc{bqo}s, and the powerset ordering
$\ae$.

\begin{theorem}\label{thm:a}
For any set system $\L$, the following are equivalent:
\begin{enumerate}
\item \label{thm:a:1} 
the quasi-ordering $\qo{\L}=(\bigcup\L,\, \le)$  satisfies the condition
      \eqref{prop:rado:1} of
      Lemma~\ref{lem:rado}.
Namely,
\begin{enumerate}
\item $\qo{\L}$ is a \textsc{wqo}; and
\item Let $F$ be any function from $[\omega]^2$ to $\bigcup\L$. Then if
      $F(\{i,j\})<F(\{i,j+1\})$ for any $i<j<\omega$, then there are
      $i<j<k<\omega$ such that $F(\{i,j\})<F(\{j,k\})$.
\end{enumerate}

\item \label{thm:a:2}
$\finclass{\L}$ is an \textsc{fess}.

\item \label{thm:a:3}
$\hat\L$ is an \textsc{fess}.
\end{enumerate} 
\end{theorem} 
\begin{proof} Write $\qo{\L}=(\bigcup\L,
 \preceq)$. Assume the condition~\eqref{thm:a:2} is false. Then there 
 are an infinite sequence of $\M_n\subseteq \finclass{\L}$
 $(n=1,2,\ldots)$ and an infinite sequence $x_n\in \bigcup\L$ ($n=1,2,\ldots$) such
 that 
\begin{align}
 v_n :=\{x_1,\ldots, x_{n-1}\} \subseteq \bigcup \M_n \not\ni x_n. \label{u}
\end{align} 
If there are $n<m$ such that $v_n \ae v_m$, then $x_i\preceq x_{m-1}$  for
 some $i<n$. However, \eqref{u} implies $x_i \in \bigcup \M_{m-1}\not\ni
 x_{m-1}$. 
By the definition of $\preceq$,  $x_i\preceq x_{m-1}$ implies 
$x_i \in \bigcup \M_{m-1}\ni
 x_{m-1}$. A contradiction. Thus the powerset ordering $(P(\bigcup\L),
 \ae)$ is not a \textsc{wqo}. By Lemma~\ref{lem:rado},  the condition
 \eqref{prop:rado:1} is false. 

Conversely, assume the
 condition~\eqref{thm:a:1} is false. Since Lemma~\ref{lem:rado}
 implies that  the condition~\eqref{thm:a:1}
 is equivalent to the well-quasi-orderedness of the powerset ordering
 $\ae$, we have an infinite sequence $(v_i)_i$ such that for all $i<j$,
 $v_i\not\ae v_j$ but $v_i\subseteq \bigcup\L$. Then there is $y^j\in
 v_j$ such that for all $y^i \in v_i$ we have $y^i\not\preceq y^j$.
By the definition of $\preceq$, there is a sequence $(L_{i,j})_{i<j}$
 such that $y^i\in L_{i,j}\not\ni y^j$. Hence the sequence $\left(
 \bigcup_{i<j} L_{i,j}\right)_j$ is an infinite learning sequence in
 $\finclass{\L}$. So $\finclass{\L}$ is not an \textsc{fess}. Thus the condition~\eqref{thm:a:3} is false.

The equivalence between the condition~\eqref{thm:a:1} and the
 condition~\eqref{thm:a:2} can be similarly proved.\qed
\end{proof} 

\begin{example}
By \cite{114871}\cite{93373}, we have
an \textsc{fess} 
\begin{align*}
\L_1:=\left\{ \{i\}\cup \{k\;;\; k\ge j\}\ ;\ i,j\in\omega\right\},
\end{align*}
because it is the memberwise union of an \textsc{fess}
$\{\ \{k\in \omega\;;\; k\ge j\}\ ;\ j\in\omega\,\}$ and an
 \textsc{fess} $\SGL$.
 But $\Finclass{\L_1}$ is
 not an \textsc{fess} according to
 \cite[Proposition~2.1.27]{brecht09:_topol_and_algeb_aspec_of}.  The
 last assertion is an easy corollary of 
 Theorem~\ref{thm:a}, 
 because $\qo{\L_1}=(\omega, =)$ and is not a \textsc{wqo}.
\end{example}

\begin{corollary} \label{cor:b} If $\L$ is a \textsc{bess}, then both of
 $\finclass{\L}$ and $\hat \L$ are \textsc{fess}s.
\end{corollary}
\begin{proof}As $\L$ is a \textsc{bess}, the quasi-ordering $\qo{\L}$
 is a \textsc{bqo} by the definition, and hence is an
 $\omega^2$-\textsc{wqo}, by the definition of \textsc{bqo}. 
 By Lemma~\ref{lem:rado}, we have the condition~\eqref{thm:a:1} of
 Theorem~\ref{thm:a} and thus the desired conclusions.\qed\end{proof}

As the class of \textsc{bqo}s enjoys the closure properties with
respect to possibly infinitary constructions, we conjecture a following:

\begin{q}If $\L$ is a \textsc{bess}, then both of
 $\finclass{\L}$ and $\hat \L$ are \textsc{bess}s.\end{q}

We contrast our characterization of set systems $\L$ having an
\textsc{fess} $\finclass{\L}$, with Shinohara-Arimura's sufficient
condition~\cite{shinohara:inductive} for a set system $\L$ to have an
\textsc{fess} $\finclass{\L}$.

\begin{definition}Let $\L$  be a set system over $X$.

$\L$ is said to have a
 \emph{finite thickness}~$($\textsc{ft}$)$, provided that for any $x\in X$
 $\#\{L\in\L\;;\; x\in L\}<\infty$. 
 $\L$  is said to have 
 \emph{no-infinite-antichain property}~$($\textsc{nia}$)$, provided that $\L$ has no infinite
antichain with respect to the set-inclusion $\subseteq$. 
\end{definition} 
The set system $\SGL$ has an \textsc{ft} but not \textsc{nia}. If a set
system has an \textsc{ft}, then it is an
\textsc{fess}~\cite{shinohara:inductive}.

\begin{proposition}[\protect{\cite{shinohara:inductive}}]\label{prop:b}
If $\L$ has an \textsc{ft} and \textsc{nia}, then
 $\finclass{\L}$ is an \textsc{fess}.
\end{proposition}
However the conjunction of \textsc{ft} and \textsc{nia} is not
preserved by the operation $\finclass{(\cdot)}$.

\begin{lemma}\begin{enumerate}
\item A set system $\L_2=\{\ [i,\ \infty)\;\cap\;\Zset\   ; \ i\ge 1\}\cup\{\{0\}\}$ has an
    \textsc{ft} and \textsc{nia} but $\L_3:=\Finclass{\L_2}$ is an
      \textsc{fess} without an \textsc{ft}~(\cite{deBrechtPrivate12}).

\item The converse of Proposition~\ref{prop:b} is false. Actually,
      $\Finclass{\L_3}$ is an \textsc{fess} but $\L_3$ does not have an
      \textsc{ft}. 
\end{enumerate}
\end{lemma}

\begin{lemma}\label{lem:compare}
\begin{enumerate}
\item \label{feniaimplies}
For any \textsc{fess} $\L$, $\L$ has \textsc{nia} if and only if $(\L, \supseteq)$ is a
     \textsc{wqo}.

\item \label{aux}
If $\finclass{\L}$ is an \textsc{fess}, then $\finclass{\L}$ has \textsc{nia}.
\end{enumerate}
\end{lemma} 
\begin{proof}
\medskip\eqref{feniaimplies}
The  if-part is immediate from the definition of \textsc{wqo}s.
Assume there is an infinite descending chain $(L_i)_i\subseteq\L$ with
 respect to $\supseteq$.  Hence, $L_1\subsetneq L_2\subsetneq
 L_3\subsetneq \cdots$. By putting $x_i\in L_{i+1}\setminus L_i$, we have
$\left\langle \langle x_1, L_2\rangle, \, \langle x_2,
 L_3\rangle,\cdots\right\rangle$ is
 an infinite learning sequence. This
contradicts that $\L$ is an  \textsc{fess}. 

\medskip\eqref{aux}
Suppose $\left(\bigcup \M_n\right)_n$ $(\M_n\subseteq \L\ ;\ n=1,2,\ldots)$
 is an infinite antichain in $\finclass{\L}$. Then 
$\left(\bigcup_{n<m}\bigcup \M_n\right)_m$ is a strictly ascending chain
 in $\finclass{\L}$. But $\L$ is an \textsc{fess}.
\qed
\end{proof}

Following relation holds among (continuous) deformations,
\textsc{nia} and \textsc{ft}:
\begin{lemma}
\begin{enumerate}
\item If a set system $\L$ has
     \textsc{nia}, so does $\op[\L]$ of $\L$ for any 
				     deformation $\op$.

\item For any nonempty set $X$ and for any $x\in X$, a function $\op :P(X)\to P(X)\;;\; A\mapsto
      A\cup\{x\}$ is a continuous deformation. Thus even if $\L$
      has an \textsc{ft},		    $\op[\L]$ does not. 
\end{enumerate} 
\end{lemma}
\begin{proof} (1) 
If the image $\{ \op  ( L_i )  \}_i$ of 
$\{ L_i \}_i  \subseteq  P(T)$  by a
deformation $\op $ is an infinite antichain,
then, for any distinct $i, j\in\omega$ there exists $n_{i,j} \in \op ( L_i
 ) \setminus  \op ( L_j )$. Let $\op $ be as in the equation \eqref{eq:mono}. Then we have 
$\exists  v_{i,j} \in P( L_i) . \;   R(n_{i,j},\, v_{i,j})$ and
$\forall v \in P( L_j). \;  \neg
			     R(n_{i,j}, v)$.
Therefore 
$  v_{i,j}$ is not a subset of $L_j$.
Thus, there exists $a_{i,j}\in v_{i,j}\setminus L_j\ \subseteq
\  L_i \setminus  L_j$.
Hence, $\{ L_i \}_i$ is also an infinite antichain. 
(2) It is immediate.
\qed
\end{proof}

Finally, we remark that a condition for a set system $\L$ to
satisfy $\finclass{\L}$ being an \textsc{fess} does not depend on the
structure of $\L$ with respect to the set-inclusion, in view of the
assertion~\eqref{assert:rado} and the assertion~\eqref{ppfomegafess} of
following Lemma~\ref{lem:1}.

We recall that a quasi-ordering $\X=(X,\,\preceq)$ is a \textsc{wqo}, if
and only if any upper-closed subset  of $X$ is a finite union of
principal filters~\cite{MR0049867}. 

\begin{definition}
For a quasi-ordering $\X=(X,\, \preceq)$, define  $\ppf{\X}$ to
 be the set of principal filters of $\X$. Let 
 $\PF{\X}$ be $\ppf{\X}$ ordered by the reverse set-inclusion. 
\end{definition}

\begin{fact}\label{fact:repre}Let  $\X$ be a quasi-ordering.
\begin{enumerate}
\item \label{qo:wqopf} $\closure{\X}=\ppf{\X}^{<\omega}$ if and only if $\X$ is a \textsc{wqo}.
\item \label{qo:ppf} $\X=\qo{\ \ppf{\X}\ }$.

\item \label{pf:repre} $\X$ is order-isomorphic to $\PF{\X}$ for any partial ordering $\X$.

\item $\dim \ppf{\X}\ \le\ \dim\closure{\X}=\otp(\X)$. 

\item For the partial
      order 
      \begin{align*}\X=\left( \{b\}\cup\{a_i\;;\; i\in \omega\},\ \{(b,\ a_i)\;;\;
      i\in\omega\}\right),
\end{align*}
 we have $\dim \PF{\X}=1$ but $\dim \X=\infty$.
\end{enumerate}
\end{fact}

\begin{lemma}\label{lem:1}
\begin{enumerate}
\item \label{cfess} $\L_1$  has \textsc{nia}.

\item \label{assert:rado} A quasi-ordering $(\L_1,\ \supseteq)$ is order-isomorphic to
      $\PF{\; (\L_1,\,\supseteq)\; }$. They are \textsc{wqo}s.

\item \label{nobqo} None of $(\L_1,\ \supseteq)$ and $\PF{\; (\L_1,\,\supseteq)\; }$ is  a \textsc{bqo}.

\item \label{nobess} None of $\L_1$ and $\ppf{\; (\L_1,\,\supseteq)\; }$ is a \textsc{bess}.

\item \label{ppfomegafess} $\Finclass{\L_1}$ is not an
      \textsc{fess}, but
$\ppf{\; (\L_1,\,\supseteq)\; }^{<\omega}$ is.
\end{enumerate}
\end{lemma}
\begin{proof}\eqref{cfess}
By \cite[Proposition~3.3]{1715965}). 
\medskip

\eqref{assert:rado}
By Lemma~\ref{lem:compare}~\eqref{feniaimplies} and the assertion \eqref{cfess}
 of this Lemma, the partial ordering $(\L_1,\supseteq)$ is a \textsc{wqo}.

\medskip
\eqref{nobqo} It is because a following function $F$ does not satisfy Lemma~\ref{lem:rado}~\eqref{prop:rado:1}~(b):
\begin{align*}
F(\{i,j\}):=\{i\}\cup \{k\;;\; k> j\}.\quad (i<j)
\end{align*}
\medskip
\eqref{nobess}  
Let $(\omega, \preceq)$ be $\qo{\L_1}$. Then $n\not\preceq m$ ($n<m$)
 because $n\in F(\{n,m\})\not\ni m$, while $n\not\preceq m$ ($n>m$)
 because
$n\in F(\{n, n+1\})\not\ni m$. Therefore $\qo{\L_1}$ is not a
 \textsc{wqo}, hence is not a \textsc{bqo}. So $\L_1$ does not satisfy
 the condition~\ref{thm:a:1} of Theorem~\ref{thm:a}. Hence
 $\Finclass{\L_1}$ is not an \textsc{fess}.

By Fact~\ref{fact:repre}~\eqref{qo:ppf} and the assertion \eqref{assert:rado} of this
 Lemma, 
$\qo{\ppf{\; (\L_1,\,\supseteq)\; }}$ is $(\L_1,\,\supseteq)$ which is not a
 \textsc{bqo} by the assertion~\ref{nobqo} of this Lemma.

\medskip
\eqref{ppfomegafess}
Since $\qo{\L_1}$ is not a \textsc{wqo} by the proof of the assertion \eqref{nobess} of
 this lemma,  $\Finclass{\L_1}$ is not an \textsc{fess} because of Theorem~\ref{thm:a}.
By Fact~\ref{fact:repre}~\eqref{qo:wqopf}, the set system $\ppf{(\L_1,
 \supseteq)}^{<\omega}$ is $\closure{(\L_1, \supseteq)}$,  which is an
 \textsc{fess} by the assertion \eqref{assert:rado} of this Lemma.
\qed \end{proof}

Although $\ppf{(\L_1, \supseteq)}$ is not a \textsc{bess}, it satisfies the
condition~\ref{thm:a:1} of Theorem~\ref{thm:a}, according to the
assertion \eqref{ppfomegafess} of Lemma~\ref{lem:1}.

\bigskip
\bigskip \noindent{\sf Acknowledgement.}

\noindent
This work is partially
supported by Grant-in-Aid for Scientific Research~(C)~(21540105) of the
Ministry of Education, Culture, Sports, Science and Technology
(MEXT). The author thanks Hiroki Arimura, Masami Ito, Makoto Kanazawa, Mizuhito Ogawa, and anonymous
referees.  Special thanks go to Ken-ichi Kawarabayashi for his
encouragement after the 11-March earthquake.

\appendix
\section{Ramsey's numbers for well-founded
 trees and order type of set systems}\label{sec:Ramsey}

Let $\X_i$ be a quasi-ordering with the maximal order type
$\otp(\X_i)<\omega$ and $\L_i$ be a
set system with the order type $\dim{\L_i}<\omega$ ($i=1,2$). 
Let $\Ram(n,m)$ be the Ramsey number~\cite{Graham.Rothschild.ea:80} of
$n$ and $m$. Then we prove
\begin{enumerate}
\item \cite[Lemma~6]{AkamaSetSystem} For
the memberwise union
$\L_1\ewunion \L_2=\{L_1\cup L_2\;;\; L_1\in \L_1,\ L_2\in L_2\}$,
\begin{align*}
\dim (\L_1\ewunion \L_2) +1 < \Ram(\dim \L_1+2, \dim\L_2+2). 
\end{align*} 

\item  \cite[Theorem~8]{AkamaSetSystem} 
$\otp\left(\X_1\cap\X_2\right)<\Ram(\otp(\X_1)+1,\ \otp(\X_2)+1)$. 
\end{enumerate}

We wish to generalize these two for the case $\dim \L_i$ $(i=1,2)$ being
general ordinal numbers. To directly generalize the proof argument of
the two, we pose a following question. 
By a tree, we mean a prefix-closed set of possibly infinite sequences.  
A well-founded tree is, by definition, a tree with all the elements
 being finite sequences.

\begin{q}Is there a reasonably simple, ordinal binary (partial) function $F$ on ordinal numbers such that
``for all ordinal numbers $\beta$ and $\gamma$ there exists an ordinal number
 $\alpha\le F(\beta,\gamma)$ with a following property: for any coloring
 of any well-founded tree $T_0$ of
order type $\alpha$ with red and black,  either there is a
well-founded tree $T_1$ of order type $\beta$ such that $T_1$ is
homeomorphically embedded into the red nodes of $T_0$, or there is a
well-founded tree $T_2$ of order type $\gamma$ such that $T_2$ is
homeomorphically embedded into the black nodes of $T_0$.'' 
\end{q}

\section{Initial segments of quasi-ordering : computability theoretic view} \label{sec:wsr}

For every nonnegative integer $z$, a set $\{z_1,\ldots,
 z_m\}$ of nonnegative integers $z_1,\ldots,z_m$ with
 $z=2^{z_1}+\cdots +2^{z_m}$ is denoted by $E_z$.
\begin{definition}
A set $A\subseteq\omega $ is called \emph{positively reducible} via a
 recursive function $f:\omega\to\omega$ to
 $B\subseteq\omega $ ($A\le_p B$ via $f$, in symbol), provided that for all $x$, $x\in A$ if
 and only if there exists $y\in E_{f(x)}$ such that $E_y\subseteq
 B$. 
 Intuitively, a finite set $E_y$ means a conjunction of Boolean
 variables, and a finite set $E_{f(x)}$ means a disjunction of such
 conjunctions $E_y$ over $y\in E_{f(x)}$. We write $A\le_p B$ if there
 exists a recursive function $f:\omega\to\omega$ such that $A\le_p B$ via $f$.
\end{definition}

We observe that  
for any recursive relation $R\subseteq\omega\times\inhfin{\omega}$ and
for any class $\L\subseteq P(\omega)$, the image $\injinv \op_R[\inj \L]$ is the
     class of sets  \emph{positively reducible}~\cite{MR0220595} to some sets in $\L$
     ``uniformly'' via a single recursive
     function 
\begin{align*}
f_R (x)= \sum_{R(x,v)}2^{\sum_{i\in v}2^i}.
\end{align*}

According to \cite{MR0220595}, the class of \emph{semirecursive sets} is
closed by the positive reduction~(equivalent to effective continuous
deformation, in spirit), and a semirecursive set is exactly an
\emph{initial segment} of some recursive linear ordering on $\omega$.

\begin{definition}
A set $M\subseteq\omega $ is called \emph{semirecursive}~\cite{MR0220595} if there is a recursive
function $\psi$ of two variables such that 
\begin{align}
&(x \in M \wedge y \not\in M) \vee (x\not\in M\wedge
 y\in M)\nonumber\\
& \implies \psi(x,y)\in  \{x, y\} \cap M
. \label{semirec}
\end{align} 
In \cite{MR1056377}, Jockusch and Owings introduced a following
generalization of a semirecursive set:
$M\subseteq\omega $ is \emph{semi-r.e.\@} if and only if there exists a partial recursive function $\psi$ of two variables
such that for all $x, y \in\omega$
\begin{displaymath} \Bigl(
x \in M \vee  y\in M \implies \psi(x,y)\in \{x, y\} \cap M\Bigr).
\end{displaymath} 
Furthermore, they introduced a following generalization of a semi-r.e.\@ set:
$M$ is \emph{weakly semirecursive} if and only if there exists a partial recursive function $\psi$ of two variables
such that the condition~\eqref{semirec} holds. 

The (partial) function
$\psi$ is called a \emph{selector function} of the
semirecurisve~(semi-r.e., weakly semirecursive) set $M$.
\end{definition}

We adapt the notion of the initial segments of partial
orderings~\cite[p.~136]{MR1215090}, as follows:

\begin{definition}\label{def:initial}
For any quasi-ordering $\preceq$ on $\omega $, we say $M\subseteq\omega $ is
an \emph{initial segment} of $\preceq$, if and only if any of $M$ is strictly
smaller with respect to $\preceq$ than any of the complement $\overline{M}$.
\end{definition}

Every initial segment of a quasi-ordering is trivial, if and only if the
 undirected graph induced by 
 the quasi-ordering is not connected.
A non-trivial initial segment may have
downward branching.

We characterize a weakly semirecursive sets and semi-r.e.\@ sets by
initial segments of quasi-orderings.
\begin{theorem}\label{thm:qo} A set
 $M$ is weakly semirecursive if and only if $M$ is an initial segment of an
      r.e.\@ quasi-ordering.
\end{theorem}
\begin{proof}
 ($\Rightarrow$) By
 \cite[Theorem~4.1]{MR1215090}. ($\Leftarrow$) Let the witnessing quasi-ordering be
      $\le $. Put
\begin{equation}
 \psi(x,y):=\left\{ \begin{array}{ll}
	     x, \qquad&(x\le  y\ \mbox{and}\ x\ne y);\\
	     y, &(y\le  x\ \mbox{and}\ x\ne y);\\
		     \uparrow, &\mbox{otherwise}.
\end{array}
	    \right.\label{newpsi}
\end{equation}
Then $\psi$ is clearly a partial recursive function. Assume $x\in M\not\ni
      y$. Because $M$ is an initial segment of $\le $ in a sense of Definition~\ref{def:initial}, we have
      $x\le  y$ and $x\ne y$. By the definition of $\psi$, we
      have $\psi(x,y)=x$. On the other hand, assume $x\not\in M\ni
      y$. Then $y\le  x$ and $x\ne y$. To sum up, $\psi(x,y)\in
      \{x,y\}\cap M$. Thus $M$ is a weakly semirecursive set with
      $\psi$ being a selector function.\qed
\end{proof} 

We can prove a similar result for semi-r.e.\@ sets.

\begin{theorem} A set $M$ is semi-r.e.\@ if and only if $M$ is a linearly
 ordered initial segment of an r.e.\@ quasi-ordering. 
\end{theorem}
\begin{proof}
 Only if-part is by
      \cite[Theorem~5.1]{MR1215090}. To prove the converse, assume $x\in
      M$ without loss of generality. When $y\in M$, we have $x\le  y$ or $y\le  x$
      because $M$ is linearly ordered. By \eqref{newpsi}, we have
      $\psi(x,y)\in \{x,y\}\cap M$. When $y\not\in M$, $x\le  y$
      and $x\ne y$ because $M$ is an initial segment of
      $\le$. By \eqref{newpsi}, we have $\psi(x,y)=x$.\qed
\end{proof}

A lemma similar to ``If $A\le_p B$ and $B$ is semirecursive, then $A$ is
semirecursive''~\cite[Theorem 4.2]{MR0220595} holds for semi-r.e. sets
and weakly semirecursive sets.

\begin{lemma}\label{lem:positive_red}If $A\le_p M$ and $M$ is semi-r.e.~(weakly
 semirecursive, resp.), then so is $A$.\end{lemma}

\begin{proof}
Let $\psi$ be a selector function of $M$. Because $A\le_p M$, the set $A$ is many-one reducible to $M$,  by
 \cite[Theorem~4.2~(ii)]{MR0220595}. So there exists a recursive function $g$ such that 
\begin{equation}
x\in A\iff
 g(x)\in M. \label{reduction}
\end{equation}
Define a partial recursive function  $\psi'$ by
\begin{equation}
\psi'(x,y)=\left\{\begin{array}{ll}
	   x,\qquad &(\psi(g(x),\, g(y)) = g(x));\\
	    y, &(\psi(g(x),\, g(y)) = g(y));\\
		   \uparrow, & (\mbox{otherwise}).
\end{array}\right. \label{newselector}
\end{equation}
(i) Assume $M$ is weakly semirecursive. Suppose $x\in A\not\ni y$ without
 loss of generality. By \eqref{reduction},
 $g(x)\in M\not\ni g(y)$.  Thus $\psi(g(x), g(y))\in \{ g(x), g(y)\}\cap
 M$. By $g(y)\not\in M$, we have $\psi(g(x),\, g(y))=g(x)\in M$. Hence
 $\psi'(x,\, y)=x\in \{x,y\}\cap A$. Therefore $A$ is weakly semirecursive
 with a selector function $\psi'$.

(ii) Assume $M$ is semi-r.e. Suppose $x\in A$ or $y\in A$. Then $g(x)\in M$
 or $g(y)\in M$. So $\psi(g(x),\, g(y))\in \{g(x),\, g(y)\}\cap M$. When
 $\psi(g(x),\, g(y))=g(x)$, by a similar argument of (i), we have
 $\psi'(x,y)\in \{x,y\}\cap A$. When $\psi(g(x),\, g(y))=g(y)$,
 $\psi'(x,y)=y\in\{x,y\}\cap A$. Thus $A$ is semi-r.e.\@ with a selector
 function $\psi'$.\qed
\end{proof}

\begin{corollary}\label{cor:image}Let $R\subseteq \omega \times\inhfin{\omega }$ be a recursive relation. Then  if $A$ is semirecursive (semi-r.e., weakly semirecursive resp.),
then so is $B\subseteq \omega $ where $B=\op_R(A)$.\end{corollary}

\section{A new order type of a set system}\label{sec:preliminary}

\begin{definition}\label{def:dimfess}
A \emph{learning sequence} of a set system $\L\subseteq P(T)$ is, by
definition, a possibly infinite sequence $$\left\langle \langle t_0, A_1\rangle, \langle t_1,
A_2\rangle ,\ldots\Bigl(, \langle t_n,
A_{n+1}\rangle\Bigr) \right\rangle$$ such that for each $i<n$
$\{t_0,\ldots,t_i\}\subseteq A_{i+1}\in \L$. In
particular, we call the sequence \emph{bad} if $A_{i+1}\not\ni t_{i+1} $
for each $i$.

We say a set system $\L\subseteq P(T)$ has \emph{infinite elasticity},
provided that 
there are infinite bad learning sequences. Otherwise, we say $\L$ has an
\emph{\textsc{fe}}, and call $\L$ an \textsc{fess}.

Let $T$ be a well-founded tree. For each node $\sigma$ of $T$, let the
ordinal number $| \sigma |$ be the supremum of $|\sigma'| + 1$ such that
$\sigma'\in T$ is an immediate extension of $\sigma$. Then the
\emph{order type} $|T|$ of the well-founded tree $T$ is defined by the
ordinal number $|\emptyseq|$ assigned to the root $\emptyseq$ of $T$.
For a tree $T$ which is not well-founded, let $|T|$ be $\infty$. 

The order type of $\L$, denoted by $\dim \L$, is, by definition, the
 order type of the tree of bad learning sequences of $\L$. 
\end{definition}

In the premise of Proposition~\ref{prop:u},  we cannot replace the domain of the
continuous function $\op:\{0,1\}^{\bigcup\M}\to \L$ with a set system $\M$. We have 
following counterexample: $\M=\{\;\{i\}\;;\;i\in\omega\}$ is a discrete subspace of
the product topology $\{0,1\}^\omega$ and hence any function from the
relative topology $\M$ to a set system $\L$ is continuous even if $\L$ is
not an \textsc{fess}.

\end{document}